\documentclass[12pt]{article}

\usepackage[margin=1in]{geometry}
\usepackage{graphicx}             
\usepackage{amsmath}               
\usepackage{amsfonts}              
\usepackage{amsthm}                
\usepackage{amssymb}
\usepackage{amscd}
\usepackage{extarrows}
\usepackage[shortlabels]{enumitem}
\usepackage{secdot}
\usepackage{tikz}
\usepackage{sectsty}

\sectionfont{\center \normalsize}
\subsectionfont{\center \normalsize}
\subsubsectionfont{\normalsize}
\paragraphfont{\normalsize}

\sectiondot{subsection}

\newtheorem{thm}{Theorem}[section]

\newtheorem{lem}[thm]{Lemma}
\newtheorem{prop}[thm]{Proposition}

\theoremstyle{definition}

\theoremstyle{remark}

\theoremstyle{remark}

\theoremstyle{remark}

\theoremstyle{remark}

\theoremstyle{remark}

\theoremstyle{remark}

\DeclareMathOperator{\Ad}{Ad}
\DeclareMathOperator{\Aut}{Aut}
\DeclareMathOperator{\Br}{Br}
\DeclareMathOperator{\id}{id}
\DeclareMathOperator{\ind}{ind}

\DeclareMathOperator{\Lie}{Lie}

\DeclareMathOperator{\Int}{Int}

\DeclareMathOperator{\PGL}{PGL}

\DeclareMathOperator{\Sim}{Sim}

\DeclareMathOperator{\GL}{GL}

\DeclareMathOperator{\Q}{\mathbb{Q}}
\DeclareMathOperator{\G}{\mathbb{G}}

\makeatletter
\newcommand{\etale}{\'etal\@ifstar{\'e}{e\space}}
\makeatother
\makeatletter
\newcommand{\CT}{Colliot-Th\'el\`en\@ifstar{\'e}{e}}
\makeatother


\newcommand*{\TitleFont}{%
      \usefont{\encodingdefault}{\rmdefault}{b}{n}%
      \fontsize{14}{20}%
      \selectfont}

\newcommand*{\AddFont}{%
      \usefont{\encodingdefault}{\rmdefault}{}{n}%
      \fontsize{10}{20}%
      \selectfont}

\begin{document}

\nocite{*}

\title{\TitleFont \textbf{TOTARO'S QUESTION FOR ADJOINT GROUPS \\ OF TYPES $A_{1}$ and $A_{2n}$}}
\author{REED LEON GORDON-SARNEY}
\date{\AddFont DEPARTMENT OF MATHEMATICS \& COMPUTER SCIENCE \\ EMORY UNIVERSITY, ATLANTA, GA 30322 USA}

\maketitle

\begin{abstract}
Let $G$ be a smooth connected linear algebraic group over a field $k$, and let $X$ be a $G$-torsor.
Totaro asked: if $X$ admits a zero-cycle of degree $d \geq 1$, then does $X$ have a closed \etale point of degree dividing $d$?
We give an affirmative answer for absolutely simple classical adjoint groups of types $A_1$ and $A_{2n}$ over fields of characteristic $\neq 2$.
\end{abstract}
\vspace{0.15cm}
\section{Introduction}
Given a variety $X$ over a field $k$, one can define its \emph{index} by \[ \ind(X) := \gcd\{[L:k] : L/k \textup{ is a finite field extension such that } X(L) \neq \varnothing \}. \]
The index of a variety equals the minimal positive degree of a zero-cycle of its closed points, and it is natural to ask: does a variety admit a closed point of degree equal to its index?
It is well-known that the answer is negative for general varieties; for example, Parimala \cite{pari05} produced a projective homogeneous space under a smooth connected linear algebraic group over $\Q_p((t))$ admitting a zero-cycle of degree 1 with no rational points.
One would hope that the question would have an affirmative answer for some nice class of varieties.

Serre originally raised the question in the case of principal homogeneous spaces (or torsors) of index 1 under smooth connected linear algebraic groups \cite{serre95}.
Since such a $G$-torsor $X$ over $k$ has a rational point if and only if its corresponding cohomology class $[X] \in H^1(k,G)$ is trivial, we can phrase Serre's question in the language of Galois cohomology.
\\\\
\noindent
\textbf{Serre's Question}. Let $G$ be a smooth connected linear algebraic group over a field $k$, and let $L_1,\ldots,L_m/k$ be finite field extensions with $\gcd\{[L_i:k]\} = 1$.
Does the natural map \[ H^1(k,G) \to \prod H^1(L_i, G_{L_i}) \] have trivial kernel?
\\

To date, no counterexamples are known, and the literature contains affirmative proofs in many special cases (e.g., Bayer--Lenstra \cite{bayerlenstra90}, Bhaskhar \cite{niv16}, and Black \cite{black11a,black11b}).
Refer to Parimala \cite{pari05} for a more comprehensive review of what is known on the question.
Totaro generalized Serre's question, asking about closed \etale points on torsors of \emph{arbitrary} index under smooth connected linear algebraic groups \cite{totaro04}.
\\\\
\noindent
\textbf{Totaro's Question}. Let $G$ be a smooth connected linear algebraic group over a field $k$, and let $[X] \in H^1(k,G)$. 
Is there a separable field extension $F/k$ with $[F:k] = \ind(X)$ such that $[X_F] = 1 \in H^1(F,G_F)$?
\\

While it is expected that the answer to Totaro's question is `yes' in general, affirmative proofs in special cases are extremely rare (cf. Black--Parimala \cite{blackpari14}, Totaro \cite{totaro04}, Garibaldi--Hoffman \cite{garihoff06}, and G.-S. \cite{totori}).
This paper extends what little is known on Totaro's question with a positive answer for an infinite class of linear algebraic groups.

Let us proceed with some notation and definitions.
Fix a field $k$ of characteristic $\neq 2$, let $K/k$ be an \etale quadratic extension, and let $A$ be a central simple algebra over $K$.
An antiautomorphism $\sigma$ on $A$ is called an \emph{involution} if $\sigma^2 = \id$; it is called an involution \emph{of the first kind} if $[K : K^\sigma] = 1$ and \emph{of the second kind} or \emph{unitary} if $[K : K^\sigma] = 2$.
Suppose $\sigma$ is unitary with fixed field $K^\sigma = k$.
For clarity, we call $\sigma$ a \emph{$K/k$--involution}.
Define the automorphisms of $(A,\sigma)$ to be the $K$-automorphisms of $A$ that commute with $\sigma$.
Then \[\Aut(A,\sigma)(k) \cong \{ \Int(a) \in \Aut_K(A) : a \in A^\times, \sigma(a)a \in k^\times \},\] where $\Int(a) : A \to A$ is given by $\Int(a)(x) = axa^{-1}$.
The elements $a \in A^\times$ such that $\sigma(a)a \in k^\times$, called the \emph{similitudes} of $(A,\sigma)$, form a group denoted $\Sim(A,\sigma)(k)$; it is clear that they only determine the automorphisms of $(A,\sigma)$ up to scalars from $K^\times$.
Viewed functorially, we have a short exact sequence of linear algebraic groups over $k$ \[ 1 \to R_{K/k}\G_m \to \Sim(A,\sigma) \xrightarrow{\Int} \Aut(A,\sigma) \to 1.\]

Linear algebraic groups with trivial center are said to be \emph{adjoint}.
Adjoint groups appear as images of adjoint representations $\Ad : G \to \Aut(\Lie(G))$ where $\Lie(G)$ is the Lie algebra associated to a linear algebraic group $G$.
The classification of \emph{absolutely simple} (i.e., simple over an algebraic closure) linear algebraic groups separates \emph{classical} groups from \emph{exceptional} groups where absolutely simple classical groups are classified into types $A_n$, $B_n$, $C_n$, and $D_n$ (non-trialitarian $D_4$ excluded).
By work of Weil, classical adjoint groups can be interpreted in the language of algebras with involution; in particular, an absolutely simple classical adjoint group of type $A_n$ over $k$ is isomorphic to $\Aut(A,\sigma)$ for a central simple algebra $A$ of degree $n+1$ over an \etale quadratic extension $K/k$ and $\sigma$ a $K/k$-involution on $A$.

In this paper, we prove

\begin{thm}
\label{mainthm}
Let $G$ be an absolutely simple classical adjoint group of type $A_1$ or $A_{2n}$ over a field $k$ of characteristic $\neq 2$, and let $X$ be a $G$-torsor over $k$.
Then there exists a separable field extension $F/k$ of degree $\ind(X)$ such that $[X_F] = 1 \in H^1(F,G_F)$.
\end{thm}

Theorem \ref{mainthm} has a concrete interpretation in terms of algebras with unitary involution.
Let $K/k$ be an \etale quadratic extension, let $A$ and $B$ be central simple algebras over $K$ of degree 2 or odd degree, and let $\sigma$ and $\tau$ be $K/k$-involutions on $A$ and $B$.
If $L_1,\ldots,L_m/k$ are finite field extensions with $\gcd\{[L_i:k]\} = d$ such that $(A,\sigma)_{L_i} \cong (B,\tau)_{L_i}$ for $i = 1,\ldots,m$, then there is a separable field extension $F/k$ with $[F:k] \mid d$ such that $(A,\sigma)_F \cong (B,\tau)_F$.

\section{Preliminaries}

Proceeding with the notation from above, let $K/k$ be an \etale quadratic extension, let $A$ be a central simple algebra over $K$, and let $\sigma$ be a $K/k$--involution on $A$.
If $A$ is Brauer--equivalent to a division algebra $D$, then $D$ also admits a unitary involution $\delta$ by the existence criterion:

\begin{thm}[Albert--Riehm--Scharlau \cite{schar75}, pp. 31]
\label{cores}
A central simple algebra $D$ over $K$ admits a $K/k$--involution if and only if $[D] \in \ker[\Br K \xrightarrow{\textup{cores}} \Br k]$.
\end{thm}

Since Totaro's question asks about the existence of a \emph{separable} field extensions over which a given torsor has a point, the following classical theorem will prove essential.

\begin{thm}[Jacobson \cite{jacobson96}, Theorem 5.3.18]
\label{max}
Let $D$ be a central division algebra over $K$ with $K/k$--involution $\delta$.
Then there exists a maximal subfield $E \subseteq D$, separable over $k$, such that $\delta(E) = E$ and $E = KE^\delta$.
\end{thm}

If $G \cong \Aut(A,\sigma)$ is absolutely simple and adjoint of type $A_n$, then $H^1(k,G)$ classifies isomorphism classes of algebras of degree $n+1$ over $K$ with unitary involution.
Since this Galois cohomology set has trivial element $[(A,\sigma)]$, for any field extension $L/k$,
\[ \begin{array}{rcl}
[(B,\tau)]_L = 1 \in H^1(L,G_L) & \Leftrightarrow & (A,\sigma) \otimes_k L \cong (B,\tau) \otimes_k L \\
& \Leftrightarrow & A \otimes_k L \cong B \otimes_k L \textup{ and } \sigma \otimes \id_L \cong \tau \otimes \id_L \\
& \Leftrightarrow & L \textup{ splits } A \otimes_K B^\textup{op} \textup{ and } \sigma \otimes \id_L \cong \tau \otimes \id_L.
\end{array} \]
Our objective then is to find minimal separable field extensions of $k$ that split $A \otimes_K B^{\textup{op}}$ followed by minimal separable field extensions to make the involutions isomorphic.
In fact, $\sigma$ is the adjoint involution of some hermitian form on $(D,\delta)$ determined up to similarity in $k^\times$, and two unitary involutions on $A$ are isomorphic if and only if their associated hermitian forms are similar.
So once the underlying algebras are isomorphic, it suffices to find a minimal separable field extension to make the corresponding hermitian forms similar.

Write $W(k)$ for the Witt ring of quadratic forms over $k$, and let $W(D,\delta)$ denote the Witt group of hermitian forms over $(D,\delta)$.
The tensor product of forms induces a $W(k)$--module structure on $W(D,\delta)$.
The next two claims will be critical to the proof of Theorem \ref{mainthm}.

\begin{lem}
\label{swap}
If $\sigma$ and $\tau$ are $K/k$--involutions on $A$, then $(A,\sigma) \otimes_k K \cong (A,\tau) \otimes_k K$.
\end{lem}

\begin{proof}
By Proposition 2.4 of Knus--Merkurjev--Rost--Tignol \cite{thebook}, $(A,\sigma) \otimes_k K \cong (A \times A^{\textup{op}}, \varepsilon)$ where $\varepsilon$ is the \emph{exchange} involution on $A \times A^{\textup{op}}$.
The same holds for $(A,\tau) \otimes_k K$.
\end{proof}

\begin{prop}
\label{schar}
Let $\sigma$ and $\tau$ be $K/k$--involutions on $A$, and let $L/k$ be a field extension of odd degree.
If $(A,\sigma) \otimes_k L \cong (A,\tau) \otimes_k L$, then $(A,\sigma) \cong (A,\tau)$.
\end{prop}

\begin{proof}
By the above remarks, it suffices to show that if $h$ and $h'$ are hermitian forms over $(D,\delta)$ such that $h \otimes_k L \cong \lambda(h' \otimes_k L)$ for some $\lambda \in L^\times$, then $h \cong \nu h'$ for some $\nu \in k^\times$.
We first assume that $L = k(\lambda)$ is a simple field extension of odd degree over $k$.
There is a natural embedding of modules (cf. Proposition 1.2 of Bayer--Lenstra \cite{bayerlenstra90}) \[ r^* : W(A,\sigma) \to W((A,\sigma) \otimes_k L) \] induced by the extension of scalars, and any non-vanishing $k$--linear functional $s : L \to k$ induces a homomorphism of modules called the \emph{Scharlau transfer} with respect to $s$ \[s_* : W((A,\sigma) \otimes_k L) \to W(A,\sigma),\] sending a class of hermitian forms $[\eta]$ on $D \otimes_k L$ over $L$ with respect to $\delta \otimes 1$ to the class of \[ s \circ \eta : (D \otimes_k L) \times (D \otimes_k L) \xrightarrow{\eta} L \xrightarrow{s} k.\]
Arguing as in Chapter 2, Lemma 5.8 of Scharlau \cite{schar85} and Proposition 1.2 of Bayer--Lenstra \cite{bayerlenstra90}, given the linear functional defined by $s(1) = 1$ and $s(\lambda) = \cdots = s(\lambda^{[L:k]-1}) = 0$, the Scharlau transfer with respect to $s$ satisfies the projection formulas \[ s_*([h \otimes_k L]) = s_*(r^*([h])) = s_*(r^*([1])) \cdot [h] = [h] \] and \[s_*([\lambda(h' \otimes_k L)]) = s_*([\lambda] \cdot r^*([h'])) = s_*([\lambda]) \cdot [h'] = [N_{L/k}(\lambda)h']. \]
Since $h \otimes_k L \cong \lambda(h' \otimes_k L)$, comparing dimensions yields that $h \cong N_{L/k}(\lambda)h'$.

Now, if $k(\lambda) \subsetneq L$, then we can filter $L/k(\lambda)$ as a tower of simple field extensions \[ k(\lambda, \lambda_1, \ldots, \lambda_{n-1}, \lambda_n) \supsetneq k(\lambda, \lambda_1, \ldots, \lambda_{n-1}) \supsetneq \cdots \supsetneq k(\lambda), \] each of odd degree.
Let $L_0 = k(\lambda)$ and $L_i = k(\lambda, \lambda_1, \ldots, \lambda_i)$ for $i = 1,\ldots,n$.
For each field extension $L_i/L_{i-1}$ of degree $d_i$, define an $L_{i-1}$--linear functional $s^i : L_i \to L_{i-1}$ by $s^i(1) = 1$ and $s^i(\lambda_i) = \cdots = s^i(\lambda_i^{d_i-1}) = 0$.
Each of of these linear functionals is also $k(\lambda)$--linear, and so each associated Scharlau transfer satisfies $s^i_*([\lambda]) = [\lambda]$ by the projection formulas.
Then \[ s^i_*([h \otimes_k L_i]) = [h \otimes_k L_{i-1}] \] and \[ s^i_*([\lambda(h' \otimes_k L_i]) = [\lambda(h' \otimes_k L_{i-1})]\] for each $i = 1,\ldots,n$.
By comparing dimensions, the result is immediate.
\end{proof}

\section{Proof of Theorem \ref{mainthm}}

The \etale quadratic extension $K$ is isomorphic to $k \times k$ or a quadratic field extension of $k$.
\\\\
\emph{Case 1.} $K \cong k \times k$.
\\\\
\indent 
In fact, Totaro's question has a positive answer for adjoint groups of type $A_n$ for \emph{any} $n$ in this case.
Proposition 2.4 of Knus--Merkurjev--Rost--Tignol \cite{thebook} tells us that $(A,\sigma) \cong (B \times B^{\textup{op}},\varepsilon)$ where $B$ is a central simple algebra of degree $n+1$ over $k$ and $\varepsilon$ is the exchange involution on $B \times B^{\textup{op}}$.
So $G \cong \Aut(A,\sigma) \cong \PGL_1(B)$.
Since $H^1(k,\GL_1(B)) = 1$ by a generalization of Hilbert 90, taking Galois cohomology of the short exact sequence \[ 1 \to \G_m \to \GL_1(B) \to \PGL_1(B) \to 1 \]  of linear algebraic groups over $k$ yields an injection $H^1(k,\PGL_1(B)) \hookrightarrow \Br k$.
A $\PGL_1(B)$--torsor is a Severi--Brauer variety $X$ associated to some central simple algebra $C$ of degree $\deg(B)$ over $k$, and the injection $H^1(k,\PGL_1(B)) \hookrightarrow \Br k$ is given by $[X] \mapsto [C \otimes_k B^{\textup{op}}]$.
So $\ind([X]) = \ind_{\textup{Sch}}(C \otimes_k B^{\textup{op}})$ where $\ind_{\textup{Sch}}(C \otimes_k B^{\textup{op}})$ denotes the \emph{Schur index} of $C \otimes_k B^\textup{op}$, the degree of its Brauer--equivalent division algebra and therefore the minimal degree of a separable splitting field for $C \otimes_k B^\textup{op}$ (cf. Proposition 4.5.4 from Gille--Szamuely \cite{gisz06}).
The $\PGL_1(B)$-torsor $X$ then has a point over this field, as desired.
\\\\
\emph{Case 2.1.} $K/k$ is a separable quadratic field extension and $G$ is adjoint of type $A_1$.
\\\\
\indent
$G \cong \Aut(A,\sigma)$ where $A$ is a quaternion algebra over $K$ and $\sigma$ is a $K/k$-involution on $A$.
The following theorem of Albert says that quaternion algebras with $K/k$--involutions are completely determined by certain quaternion subalgebras over $k$.

\begin{thm}[Albert \cite{albert61}, pp. 61]
\label{frombelow}
Let $Q$ be a quaternion division algebra over $K$ with $K/k$--involution $\sigma$.
Then there exists a unique quaternion division subalgebra $Q_0 \subseteq Q$ over $k$ with its canonical (symplectic) involution $\sigma_0$ such that $Q \cong Q_0 \otimes_k K$ and $\sigma \cong \sigma_0 \otimes \bar{\ \ }$ where $\textup{Gal}(K/k) = \{\textup{id}, \bar{\ \ }\}$.
\end{thm}

So there is a unique quaternion algebra $A_0$ over $k$ with canonical involution $\sigma_0$ such that $(A,\sigma) \cong (A_0,\sigma_0) \otimes_k K$.
Given any $[(B,\tau)] \in H^1(k,G)$ with descent $[(B_0,\tau_0)]$, $(A,\sigma)$ and $(B,\tau)$ are completely determined by $A_0$ and $B_0$, and for any field extension $L/k$,
\[
\begin{array}{rcl}
[(B,\tau)]_L = 1 \in H^1(L,G_L) & \Leftrightarrow & (A,\sigma) \otimes_k L \cong (B,\tau) \otimes_k L \\
& \Leftrightarrow & A \otimes_k L \cong B \otimes_k L \textup{ and } \sigma \otimes \id_L \cong \tau \otimes \id_L \\
& \Leftrightarrow & A_0 \otimes_k L \cong B_0 \otimes_k L \\
& \Leftrightarrow & L \textup{ splits } A_0 \otimes_k B_0.
\end{array}
\]
So the field extensions trivializing $[(B,\tau)] \in H^1(k,G)$ are precisely the splitting fields of the central simple algebra $A_0 \otimes_k B_0$.
In particular, $\ind_{\textup{Sch}}(A_0 \otimes_k B_0) = \ind([B,\tau])$.
As above, there is a separable splitting field of $A_0 \otimes_k B_0$ of degree $\ind_{\textup{Sch}}(A_0 \otimes_k B_0)$ over $k$, yielding the result.
\\\\
\emph{Case 2.2.} $K/k$ is a separable quadratic field extension and $G$ is adjoint of type $A_{2n}$.
\\\\
\indent
$G \cong \Aut(A,\sigma)$ where $A$ is a central simple algebra odd degree $2n+1$ over $K$.
Fix $[(B,\tau)] \in H^1(k,G)$, and let $D$ be the division algebra Brauer--equivalent to $A \otimes_K B^{\textup{op}}$.
If $D$ is split by some field extension $L/k$, then so is $A \otimes_K B^{\textup{op}}$, in which case $A \otimes_k L \cong B \otimes_k L$.
Then either $\sigma$ and $\tau$ become isomorphic over $L$, in which case we are done, or $\sigma$ and $\tau$ become isomorphic over $KL$ by Lemma \ref{swap}.
Since every field extension that trivializes $[(B,\tau)]$ necessarily splits $D$, we see that $\ind([(B,\tau)]) = 2^{\theta}\ind_{\textup{Sch}}(A \otimes_K B^\textup{op})$ where $\theta = 0$ or 1.

Suppose first that $\ind([(B,\tau)]) = \ind_{\textup{Sch}}(A \otimes_K B^\textup{op})$.
Since $K^\sigma = K^\tau = k$, \[\textup{cores}(D) = \textup{cores}(A)\textup{cores}(B^\textup{op}) = 0 \in \Br k.\]
So $D$ admits a unitary involution $\delta$ such that $K^\delta = k$ by Theorem \ref{cores}.
By Theorem \ref{max}, $D$ contains a maximal subfield $E$, separable over $k$, such that $\delta(E) = E$ and $E = KE^\delta$.
Since \[\ind_{\textup{Sch}}(A \otimes_K B^\textup{op}) = \deg(D) = [E:K] = [E^\delta : k]\] and $D \otimes_k E^\delta \cong D \otimes_K E$ is split, $A \otimes_k E^{\delta} \cong B \otimes_k E^{\delta}$.
Then $\ind([(B,\tau)])$ is odd as \[\ind([(B,\tau)]) = \ind_{\textup{Sch}}(A \otimes_K B^\textup{op}) \mid \deg(A \otimes_k B^\textup{op}) = (2n+1)^2.\]
So there is a field extension $L/k$ of odd degree such that $(A,\sigma) \otimes_k L \cong (B,\tau) \otimes_k L$, hence \[ ((A,\sigma) \otimes_k E^{\delta}) \otimes_{E^{\delta}} (E^{\delta} \otimes_k L) \cong ((B,\tau) \otimes_k E^{\delta}) \otimes_{E^{\delta}} (E^{\delta} \otimes_k L).\]
In particular, $\sigma$ and $\tau$ (viewed as involutions on the isomorphic algebras $A \otimes_k E^\delta$ and $B \otimes_k E^\delta$) become isomorphic over $E^\delta \otimes_k L$.
Since $[L:k]$ is odd, $E^\delta \otimes_k L$ is isomorphic to a direct product of field extensions of $E^\delta$, at least one of which must have odd degree, else $\dim_{E^\delta}(E^\delta \otimes_k L)$ would be even.
Call this extension $M$.
Then $\sigma$ and $\tau$ become isomorphic over $M$.
As $[M:E^\delta]$ is odd, $\sigma$ and $\tau$ become isomorphic over $E^\delta$ by Proposition \ref{schar}, meaning that $(A,\sigma) \otimes_k E^\delta \cong (B,\tau) \otimes_k E^\delta$.
Since $[E^\delta : k] = \ind([(B,\tau)])$, it suffices to take $F = E^\delta$.

Finally, suppose that $\ind([(B,\tau)]) = 2\ind_{\textup{Sch}}(A \otimes_K B^\textup{op})$.
Proceed exactly as above to obtain the separable field extension $E^\delta/k$ of degree $\ind_{\textup{Sch}}(A \otimes_K B^\textup{op})$ such that $A \otimes_k E^\delta \cong B \otimes_k E^\delta$.
By Lemma \ref{swap}, $(A,\sigma) \otimes_k KE^\delta \cong (B,\tau) \otimes_k KE^\delta$.
Since $\ind_{\textup{Sch}}(A \otimes_K B^\textup{op})$ is odd, \[[KE^\delta:k] = [K:k][E^\delta :k] = 2\ind_{\textup{Sch}}(A \otimes_K B^\textup{op}) = \ind([(B,\tau)]),\] and so it suffices to take $F = KE^\delta$, completing the proof. \hspace*{\fill}$\square$

\bibliographystyle{alpha}
\bibliography{totaro}

\def\cftil#1{\ifmmode\setbox7\hbox{$\accent"5E#1$}\else
  \setbox7\hbox{\accent"5E#1}\penalty 10000\relax\fi\raise 1\ht7
  \hbox{\lower1.15ex\hbox to 1\wd7{\hss\accent"7E\hss}}\penalty 10000
  \hskip-1\wd7\penalty 10000\box7}
\begin{thebibliography}{KMRT98}

\bibitem[Alb61]{albert61}
A.~A. Albert.
\newblock {\em Structure of algebras}.
\newblock Revised printing. American Mathematical Society Colloquium
  Publications, Vol. XXIV. American Mathematical Society, Providence, R.I.,
  1961.

\bibitem[BFL90]{bayerlenstra90}
E.~Bayer-Fluckiger and H.~W. Lenstra, Jr.
\newblock Forms in odd degree extensions and self-dual normal bases.
\newblock {\em Amer. J. Math.}, 112(3):359--373, 1990.

\bibitem[Bha16]{niv16}
N.~Bhaskhar.
\newblock On {S}erre's injectivity question and norm principle.
\newblock {\em Comment. Math. Helv.}, 91(1):145--161, 2016.

\bibitem[Bla11a]{black11a}
J.~Black.
\newblock Implications of the {H}asse principle for zero cycles of degree one
  on principal homogeneous spaces.
\newblock {\em Proc. Amer. Math. Soc.}, 139(12):4163--4171, 2011.

\bibitem[Bla11b]{black11b}
J.~Black.
\newblock Zero cycles of degree one on principal homogeneous spaces.
\newblock {\em J. Algebra}, 334:232--246, 2011.

\bibitem[BP14]{blackpari14}
J.~Black and R.~Parimala.
\newblock Totaro's question for simply connected groups of low rank.
\newblock {\em Pacific J. Math.}, 269(2):257--267, 2014.

\bibitem[BQM14]{blackque14}
J.~Black and A.~Qu{\'e}guiner-Mathieu.
\newblock Involutions, odd degree extensions and generic splitting.
\newblock {\em Enseign. Math.}, 60(3-4):377--395, 2014.

\bibitem[BS06]{barq06}
P.~B. Barquero-Salavert.
\newblock Similitudes of algebras with involution under odd-degree extensions.
\newblock {\em Comm. Algebra}, 34(2):625--632, 2006.

\bibitem[CTC79]{colliocoray79}
J.-L. Colliot-Th{\'e}l{\`e}ne and D.~Coray.
\newblock L'\'equivalence rationnelle sur les points ferm\'es des surfaces
  rationnelles fibr\'ees en coniques.
\newblock {\em Compositio Math.}, 39(3):301--332, 1979.

\bibitem[Flo04]{florence04}
M.~Florence.
\newblock Z\'ero-cycles de degr\'e un sur les espaces homog\`enes.
\newblock {\em Int. Math. Res. Not.}, (54):2897--2914, 2004.

\bibitem[GH06]{garihoff06}
S.~Garibaldi and D.~W. Hoffmann.
\newblock Totaro's question on zero-cycles on {$G_2$}, {$F_4$} and {$E_6$}
  torsors.
\newblock {\em J. London Math. Soc. (2)}, 73(2):325--338, 2006.

\bibitem[GS]{totori}
R.~L. Gordon-Sarney.
\newblock Totaro's question for tori of low rank.
\newblock {\em accepted to Trans. Amer. Math. Soc.}

\bibitem[GS06]{gisz06}
P.~Gille and T.~Szamuely.
\newblock {\em Central simple algebras and {G}alois cohomology}, volume 101 of
  {\em Cambridge Studies in Advanced Mathematics}.
\newblock Cambridge University Press, Cambridge, 2006.

\bibitem[Jac96]{jacobson96}
N.~Jacobson.
\newblock {\em Finite-dimensional division algebras over fields}.
\newblock Springer-Verlag, Berlin, 1996.

\bibitem[KMRT98]{thebook}
M.-A. Knus, A.~Merkurjev, M.~Rost, and J.-P. Tignol.
\newblock {\em The book of involutions}, volume~44 of {\em American
  Mathematical Society Colloquium Publications}.
\newblock American Mathematical Society, Providence, RI, 1998.
\newblock With a preface in French by J. Tits.

\bibitem[Par05]{pari05}
R.~Parimala.
\newblock Homogeneous varieties---zero-cycles of degree one versus rational
  points.
\newblock {\em Asian J. Math.}, 9(2):251--256, 2005.

\bibitem[Sch75]{schar75}
W.~Scharlau.
\newblock Zur {E}xistenz von {I}nvolutionen auf einfachen {A}lgebren.
\newblock {\em Math. Z.}, 145(1):29--32, 1975.

\bibitem[Sch85]{schar85}
W.~Scharlau.
\newblock {\em Quadratic and {H}ermitian forms}, volume 270 of {\em Grundlehren
  der Mathematischen Wissenschaften [Fundamental Principles of Mathematical
  Sciences]}.
\newblock Springer-Verlag, Berlin, 1985.

\bibitem[Ser95]{serre95}
J.-P. Serre.
\newblock Cohomologie galoisienne: progr\`es et probl\`emes.
\newblock {\em Ast\'erisque}, (227):Exp.\ No.\ 783, 4, 229--257, 1995.
\newblock S{\'e}minaire Bourbaki, Vol. 1993/94.

\bibitem[Tot04]{totaro04}
B.~Totaro.
\newblock Splitting fields for {$E_8$}-torsors.
\newblock {\em Duke Math. J.}, 121(3):425--455, 2004.

\end{thebibliography}

\end{document}